\documentclass{amsart}
\usepackage{commands}
\newcommand{\Gr}{\mathcal{G}\mathcal{R}}
\title{Cohomology of finite graded group varieties}
\author{Camil I. Aponte Rom\'an and Alberto Chiecchio}
\date{}
\dedicatory{In memory of M. Scott Osborne,\\who kindled our love for algebra and\\coined the expression \emph{Noether correspondence} for groups}

\newcommand{\Dred}{\overline{\Delta}}
\newcommand{\uExt}{\underline{\mathrm{Ext}}}
\newcommand{\LHS}{Lyndon-Hochschild-Serre}
\newcommand{\grUT}{\boldsymbol{\mathrm{UT}}}
\newcommand{\grGL}{\boldsymbol{\mathrm{GL}}}

\begin{document}
\maketitle
\tableofcontents

%----------HTHTHT---------------------
% turn off those nasty overfull and underfull hboxes
%\hbadness=10000
%\hfuzz=50pt

%\newpage
%\begingroup
%\let\clearpage\relax

\begin{abstract}{We prove that, if $A$ is a positively graded, graded commutative, local, finite Hopf algebra, its cohomology is finitely generated, thus unifying classical results of Wilkerson and Hopkins-Smith, and of Friedlander-Suslin. We do this by showing the existence of conormal elementary quotients.}
\end{abstract}

\section{Introduction}

The cohomology of a Hopf algebra $A$, denoted by $H^{*,*}(A,\kf)=\uExt^{*,*}(A,\kf)$, has a ring structure which has been extensively studied. One natural question is if such ring is finitely generated. Another natural question is, if $M$ is a Noetherian module over a ring of $A$-invariants $S$, is the module $H^{*,*}(A,M)$ Noetherian over the ring $H^{*,*}(A,S)$. Wilkerson in \cite{Wi}, and later Hopkins and Smith in \cite{HS}, answer both questions affirmatively when $A$ is finite, positively graded, graded commutative and \emph{connected}, i.e., when $A_0=\kf$. Friedlander and Suslin study the ungraded case in \cite{FS}, and show again that the first question and a slight variation of the second question have positive answer when $A=A_0$ is finite and local .

%In the direction towards the unification of these results in a single framework, the first author introduces in \cite{AR1} the notion of \emph{(affine) graded group varieties}, which are in correspondence with positively graded, graded commutative, of finite type, local Hopf algebras. If $A$ is such a Hopf algebra, then $A_0$ is necessarily a local ring. In the same work, the first author constructs the \emph{connectivization} $\kappa(A)$ of a positively graded, graded commutative Hopf algebra, which is shown to be an algebraically connected, positively graded, graded commutative Hopf algebra. Moreover, it sits in a short exact sequence of Hopf algebras
%$$
%\kf\rightarrow A_0\rightarrow A\rightarrow\kappa(A)\rightarrow\kf,
%$$
%and $A\cong A_0\otimes\kappa(A)$ as graded algebras. Therefore it seems expected that one should be able to unify the two classical results of Wilkerson and Friedlander-Suslin.

The main result of our paper is an unification of the two classical results of Wilkerson and Friedlander-Suslin.

\begin{thm}[Theorem \ref{thm:cohom1}] Let $A$ be positively graded, graded commutative, local, finite Hopf algebra over a field $\kf$. We have the following:
\begin{itemize}
\item[(H)] the cohomology ring $H^{*,*}(A,\kf)$ is a finitely generated $\kf$-algebra;
\item[(Q)] if $A\rightarrow B$ is a quotient of graded commutative Hopf algebras, $H^{*,*}(B,\kf)$ is a finite $H^{*,*}(A,\kf)$-module;
\item[(M)] if $R$ is a Noetherian ring on which $A$ acts trivially and $M$ is a finite $R$-module, $H^{*,*}(A,M)$ is a finite $H^{*,*}(A,R)$-module.
\end{itemize}
\end{thm}

When $A=A_0$ this is the result of Friedlander-Suslin, and when $A_0=\kf$ this is the result of Wilkerson and Hopkins-Smith.

\vspace{1ex}

To prove this result we proceed by induction on $\dim_{\kf}A$. One mayor technical obstacle is the following. The proof of the result of Wilkerson (and of Hopkins-Smith) is by induction and relies on the existence of cocentral elementary quotients (see section 2.1 for the definitions). This heavily depends on the fact that the Hopf algebras considered are connected, and it is no longer true in our setting, as example \ref{ex:nococentral} shows. We show that, when $A_0$ is local, we can however construct conormal elementary quotients.

\begin{prop}[Proposition \ref{prop:conormalmonogenicquotients}] Let $A$ be a finite, positively graded, graded commutative, local, Hopf algebra. Then $A$ has a conormal elementary quotient.
\end{prop}

This result is what allow us to do induction (since these quotients are conormal and not cocentral, the induction is necessarily more involved). 

We point out that, since an (affine) infinitesimal group scheme is represented by a finite, local, commutative Hopf algebra $A$, if we consider this Hopf algebra $A$ as trivially graded, the above proposition has the following immediate corollary.

\begin{cor} Let $G$ be an (affine) infinitesimal group scheme. Then $G$ has a closed normal elementary subgroup scheme.
\end{cor}

\subsection*{Acknowledgments} We would like to thank John H. Palmieri and Julia Pevtsova for the numerous useful discussions.

\section{Preliminaries}

Throughout the paper, we will fix a perfect field $\kf$ of characteristic $p>0$ (the main theorem is trivial if $\charfield\kf=0$). 

\vspace{1ex}

We will briefly recall some of the definitions and results that we will use in our paper.

\subsection{Graded group schemes and graded group varieties}

We start with some definitions and results of the first author in \cite{AR1}. The interested reader should look at \cite{AR1} and \cite{AR2} for a more exhaustive treatise.

\begin{df}[{\cite[Definition 2.1]{AR1}}] Let $\Gr$ be the category of (finitely generated) graded commutative $\kf$-algebras. A representable functor $G: \Gr \to (groups)$ is called an \emph{affine graded group scheme}, or \emph{gr-group schemes} for short. The graded algebra representing $G$ is denoted by $\kf[G]$ and is called the \emph{coordinate algebra} of $G$. 
\end{df}

We will drop the word affine from now on, as all our gr-schemes will be assumed to be affine.

Recall that a graded algebra is graded commutative if, for $a, b$ homogeneous elements in $A$, we have that $ab = (-1)^{|a||b|}ba$, where $|a|$ is the degree of $a$. As in the ungraded case, by Yoneda's Lemma there is an equivalence of categories between gr-group schemes and graded commutative Hopf algebras.

\begin{df}[{\cite[Definition 2.5]{AR1}}] We say that a gr-group scheme $G$ is a \emph{finite gr-group scheme} if $\kf[G]$ is finite dimensional. In that case we can define $\kf G$ as the graded dual of $\kf[G]$; $\kf G$ is a called the \emph{group algebra} for $G$. 
\end{df}

\begin{df}[{\cite[Definition 2.6]{AR1}}] We say that a gr-group scheme $G$ is a \emph{positive gr-group scheme} if $\kf[G]$ is positively graded; that is $\kf[G] = \bigoplus_{i \geq 0} (\kf[G])_i$. 
\end{df}

\begin{df}[{\cite[Definition 3.1]{AR1}}] Let $G$ be a gr-group scheme, and let $A = \kf[G]$. If $A_0$ is a local ring and $A$ is positively graded, of finite type (that is, each $A_i$ is finite dimensional) we say that $G$ is a \emph{graded group variety (gr-group variety)}. 
\end{df}

\begin{rk} Equivalently, in the above definition we can ask for $A$ to be positively graded, of finite type and graded local (i.e., it has a unique homogeneous maximal ideal). Indeed, if $A$ is positively graded, $A_+$ is a ideal, and therefore $A$ is graded local if and only if $A_0$ is local.
\end{rk}

To describe the Hopf algebra structure of coordinate rings of gr-group varieties it is enough to provide the comultiplication and the counit, by the following theorem.

\begin{thm}[{\cite[Theorem 3.3]{AR1}}] Let $A$ be a positively graded, graded local bialgebra of finite type, there exists an antipode map $S$ making $A$ into a graded Hopf algebra, that is, $A$ is the coordinate ring of a gr-group variety.
\end{thm}

Recall that, a graded $\kf$-algebra $A$ is \emph{algebraically connected} (or simply \emph{connected}) if $A_0=\kf$. To relate the general case to the algebraically connected case, in \cite{AR1}, the first author constructed the algebraic connectivization of a graded Hopf algebra.

\begin{df}[{\cite[Definition 3.5]{AR1}}] Let $A$ be a graded Hopf algebra. Let $\kappa(A):=A\otimes_{A_0}k$. We call $\kappa(A)$ the \emph{algebraic connectivization} of $A$.
\end{df}

\begin{thm}[{\cite[Theorem 3.6]{AR1}}] Let $A$ be a positively graded Hopf algebra. The algebraic connectivization of $A$, $\kappa(A)$, is an algebraically connected graded Hopf algebra.
\end{thm}

Moreover, the first author shows that we can recover the algebra structure of the coordinate ring $A$ of gr-group varieties for $A_0$ and $\kappa(A)$.

\begin{thm}[{\cite[Theorem 3.11]{AR1}}] Let $A$ be the coordinate ring of a gr-group variety. Then $A\cong A_0\otimes_{\kf}\kappa(A)$ as graded algebras.
\end{thm}

Finally, we will use the following notations.

\begin{df}[{\cite[Definition 2.3]{AR1}}] We denote $\kf[x_1, \ldots, x_n]^{gr}$ to be the \emph{graded polynomial ring} over $\kf$ in $n$-variable, where $x_ix_j = (-1)^{|x_i||x_j|}x_jx_i$. Note that if $\charfield(\kf) \neq 2$, then $x_i^2 = 0$ if $|x_i|$ is odd. 
\end{df}

\begin{df} An element $x$ in a Hopf algebra is called \emph{primitive} if $\Delta(x)=x\otimes 1+1\otimes x$. A Hopf algebra $E$ is \emph{elementary} if $E\cong\kf[x]^{gr}/(x^p)$, $\Delta(x)=x\otimes 1+1\otimes x$, $\e(x)=0$. A gr-group scheme represented by an elementary Hopf algebra is called \emph{elementary}.
\end{df}

\begin{rk} In the previous definition, we do not exclude the case when $|x|$ is odd and $\charfield\kf\neq2$. In that case, $E=\kf[x]^{gr}=\kf[x]/(x^2)=\kf[x]^{gr}/(x^p)$, since $x^2=0$.
\end{rk} 

\begin{rk} This is a slightly more restrictive notion of elementary than what is used in the literature. In general, elementary Hop algebras are of the form $\kf[x]^{gr}/(x^{p^e})$ for $e\geq1$. For simplicity, we will only allow $e=1$. The advantage is that, with our notion, if $E$ is elementary, so is its dual $E^*$.
\end{rk}

\begin{df} If $A$ is a positively graded Hopf algebra, and $a\in A$, we use the notation $\Dred(a):=\Delta(a)-a\otimes 1-1\otimes a$. We will call $\Dred$ the \emph{reduced} comultiplication (or \emph{non primitive} comultiplication).
\end{df}

The reason for this notation is the following. If $A$ is graded local, the counit diagram implies that $\Delta(a)=a\otimes 1+1\otimes a+\ldots$ (see the proof of \cite[Theorem 3.3]{AR1}). Thus $\Dred(a)$ is the non-primitive part of the comultiplication. If fact, $a\in A$ is primitive if and only if $\Dred(a)=0$.

We will now briefly recall the notion of Frobenius kernel (see \cite[Section 3]{AR2} for the discussion in the graded case).
 
\begin{df} If $G$ is a gr-group scheme, the Frobenius morphism $F^r:G\rightarrow G$ is a gr-group scheme homomorphism (i.e., if it is a Hopf algebra morphism $\kf[G]\rightarrow\kf[G]$). Its kernel is a gr-group scheme, called the \emph{$r$th Frobenius kernel}, and is denoted by $G_{(r)}$. 
\end{df}

\begin{rk} If $I=(f_1,\ldots,f_n)$ is a finitely generated ideal in a ring, the ideal $I^{[p^r]}$ is the ideal $I^{[p^r]}=(f_1^{p^r},\ldots,f_n^{p^r})$, \cite[Definition 2.11]{ST}. It is not hard to see that the ideal $I^{[p^r]}$ is independent of the choice of generators for $I$, \cite[Exercise 2.12]{ST}. Let $G$ be a gr-group scheme represented by $\kf[G]$ and let $I_G$ be its the augmentation ideal; then $G_{(r)}$ is represented by $\kf[G_{(r)}]=\kf[G]/I_G^{[p^r]}$. Therefore, for any gr-group scheme $G$, the Frobenius kernel is a finite gr-group variety.
\end{rk}

\subsection{Short exact sequence of Hopf algebras and cohomology}

Recall the following definitions (see \cite[\S 1.4]{Pal} for more details).

\begin{df} Let $$\kf\rightarrow B\rightarrow A\rightarrow C\rightarrow\kf$$ be a short exact sequence of graded Hopf algebras. We say that $B$ is a \emph{normal} sub-Hopf algebra of $A$ and that $C$ is a \emph{conormal} quotient of $A$.
\end{df}

\begin{df} If $\kf[G]$ and $\kf[H]$ are coordinate rings of the graded group schemes $G$ and $H$, respectively, we say that $H$ is a \emph{closed graded subgroup scheme} of $G$ if there is a surjection of graded Hopf algebras $\kf[G]\rightarrow\kf[H]$. If furthermore $\kf[H]$ is a conormal quotient, we say that $H$ is \emph{normal} in $G$. In this case, there is a short exact sequence of graded Hopf algebras $\kf\rightarrow B\rightarrow\kf[G]\rightarrow\kf[H]\rightarrow\kf$: the gr-group scheme represented by $B$ will be denoted by $G/H$, i.e., $B=\kf[G/H]$.
\end{df}

\begin{lm}\label{lm:Yoneda} If $G$ and $H$ are graded group schemes, $H$ is a closed graded subgroup scheme of $G$ if and only if, for every graded commutative $\kf$-algebra $R$, $H(R)$ is a subgroup of $G(R)$. Moreover, in this case, $H$ is normal if and only if $H(R)$ is a normal subgroup of $G(R)$ for every $R$.
\end{lm}

\begin{proof} The first assertion is immediate form Yoneda's lemma. The second assertion is a consequence of Yoneda's lemma as well, once one notices that $\kf[G/H]$ represents the graded group scheme $G(R)/H(R)$.
\end{proof}

In \cite{AR1}, the first author proved the following lemma.

\begin{lm}[{\cite[Lemma 3.7]{AR1}}] Let $A$ be a positively graded Hopf algebra. There is a short exact sequence of graded Hopf algebras
$$
\kf\rightarrow A_0\rightarrow A\rightarrow\kappa(A)\rightarrow\kf
$$
\end{lm}

The proof is based on the observation that $A_0$ can be obtained as the cotensor product $A\hspace{-2ex}\qed_{\kappa(A)}\kf$. Let us recall that, if $f:A\rightarrow B$ is a map of Hopf algebras, the cotensor product is defined to be
$$
A\hspace{-2ex}\qed_B\kf=\{a\in A\,|\,(id\otimes f)(\Delta(a))-a\otimes1=0\},
$$
and similarly for $\kf\hspace{-2ex}\qed_B A$.

\begin{thm}[{\cite[Theorem 1.4.10]{Pal}}]\label{thm:LHS} Let $G$ be a gr-group variety with $A = \kf[G]$, and let $$k\rightarrow B\rightarrow A\rightarrow C\rightarrow k$$ be a short exact sequence of graded Hopf algebras. For any graded $A$-comodules $M_1, M_2, M_3$ there is a spectral sequence with 
$$E_2^{p,q,v} =  \uExt_{B}^{p,v}(M_1, \uExt_{C}^{q, *}(M_2, M_3)) \Rightarrow \uExt_A^{p+q,v}(M_1 \otimes M_2, M_3).$$
\end{thm}

The above is called the \emph{\LHS{} spectral sequence}, or \emph{LHS} spectral sequence, for short.

If $G$ is a gr-group scheme with coordinate ring $A$, we define the cohomology $H^{*,*}(A,\kf):=\uExt_A(\kf,\kf)$. This is a bi-graded ring, see \cite[Appendix A]{AR2}. Moreover, for any $\kf$ vector space $M$, we consider the $H^{*,*}(A,\kf)$-module $H^{*,*}(A,M):=\uExt_A(\kf,M)$. 

%One crucial result that we will use is the following theorem of \cite{Wi}.
%
%
%\begin{lm}[{\cite[Lemma 3.1]{Wi}, \cite[Lemma 1.6]{FS}}]\label{WiL}
% If a first quadrant spectral sequence $\{E_r, d_r\}$ is an $R$ module for some noetherian ring $R$, and $E_2$ is a finitely generated $R$-module, then $E_\infty$, is a finitely generated $R$-module.
%\end{lm}

For completion we state the results from \cite{FS} and \cite{Wi} that we will unify. 

\begin{thm}[{\cite[Theorems 1.1 and 1.5]{FS}}]\label{FST}
Let $G$ be an infinitesimal group scheme over $\kf$. The cohomology $H^*(G,\kf)$ is finitely generated. Moreover, if $M$ is a finite $\kf$-vector space, $H^*(G,M)$ is a finite $H^*(G,\kf)$-module.
\end{thm}

\begin{thm}[{\cite[Theorem A]{Wi}, \cite[Theorem 4.13]{HS}}]\label{WiT}
If $A$ is a finite dimensional graded connected commutative Hopf algebra, then $H^{*,*}(A, \kf)$ is a finitely generated $\kf$-algebra. If $S$ is a graded Noetherian ring of $A$ invariants, and $M$ is a finite $S$-module, $H^{*,*}(A,M)$ is a finite $H^{*,*}(A,S)$-module. If $A\rightarrow B$ is a quotient of graded connected Hopf algebras, $H^{*,*}(B,\kf)$ is a finite $H^{*,*}(A,\kf)$-module.
\end{thm}

Finally, by direct computation, \cite{L2} or \cite{Wi}, it is well known that Theorem \ref{thm:cohom1} holds when $E$ is elementary. Moreover, let us recall that the cohomology in that case is, up to nilpotents, a polynomial ring in one variable.

\begin{lm}[{\cite[Proposition 2.2]{Wi} and \cite[Proposition 6.10]{AR2}}] If $E=\kf[x]^{gr}/(x^p)$ the cohomology $H^{*,*}(E,\kf)$ is 
$$
H^{*,*}(E,\kf)\cong\begin{cases} \kf[y]^{gr},\,|y|=(1,|x|), & \charfield\kf=2,\\
\kf[y,\lambda]^{gr},\,|y|=(2,p|x|),\,|\lambda|=(1,|x|),&\charfield\kf>2,\,|x|\,even,\\
\kf[y]^{gr},\,|y|=(1,|x|),&\charfield\kf>2,\,|x|\,odd.
\end{cases}
$$
\end{lm}

\begin{rk} Thus, if $\charfield\kf=2$ or if $\charfield\kf>2$ and $|x|$ is odd, $H^{*,*}(E,\kf)\cong\kf[y]$; and if $\charfield\kf>2$ and $|x|$ is even, $H^{*,*}(E,\kf)\cong\kf[y,\lambda]/(\lambda^2)$.
\end{rk}

\section{Conormal elementary quotients}

\subsection{Wilkerson's construction}

The proof of the finite generation of the cohomology ring for finite  positively graded connected Hopf algebras, see \cite{Wi} and \cite{HS}, relies on the existence of elementary cocentral quotients in positive degree. This is however no longer true in our setting, as example \ref{ex:nococentral} shows.  In this section we recall Wilkerson's construction, and we discuss where the obstacle lay in our setting. We will construct conormal quotients in section 3.2.

\vspace{1ex}

Let $A$ be a finite positively graded Hopf algebras, and let $A^*$ be its dual. If $A^*$ is connected, there are primitive elements. Let $\chi\in A^*$ be any primitive of the highest degree (among primitives). The following properties are easy to check (see \cite{Wi} or \cite{L2}).

\begin{lm}\label{lm:primitives} Let $A^*$ be a graded Hopf algebra (not necessarily connected) and let $\chi\in A^*$ be primitive. Then $\chi^p$ is primitive and, for every $\alpha\in A^*$, $[\alpha,\chi]$ is primitive. 
\end{lm}

Back to the connected case, by maximality of the degree of $\chi$, $\chi^p=0$ and $[\alpha,\chi]=0$ for every $\alpha\in(A^*)_+$ Since $(A^*)_0=\kf$, then $\kf[\chi]^{gr}/(\chi^p)$ is a central elementary sub-Hopf algebra of $A^*$, which corresponds to a cocentral  elementary quotient $A\rightarrow E$.

\vspace{1ex}

If we follow the same approach of \cite{Wi} in our setting, the first obstacle is that, since our algebras are not connected, it is not clear that there will be primitive elements in the dual. Luckily, these elements exists.

\begin{lm} Let $A$ be a positively graded Hopf algebra. Then $A^*$ has primitive elements in positive degree.
\end{lm}

\begin{proof} Since $\kappa(A)$ is a quotient of $A$, we have an inclusion $\kappa(A)^*\subseteq A^*$, and $\kappa(A)^*$ is connected (since so is $\kappa(A)$). Since there are primitive elements in $\kappa(A)^*$, there are primitive elements in positive degree in $A^*$.
\end{proof}

Let us choose then $\chi$ of maximal degree among the primitive elements in $A^*$; since in general $(A^*)_0\neq\kf$, the element $\chi$ will not be central, as the next example shows.

\begin{ex}\label{ex:nococentral} In this example we give the coordinate ring of a finite gr-group variety with no cocentral elementary quotients. Let $A=\kf[m,x]^{gr}/(m^p,x^p)$, where $\charfield\kf=p$, $|m|=0$ and $|x|=2$, $\Dred(m)=m\otimes m$, $\Dred(x)=x\otimes m$, $\e(m)=\e(x)=0$. The quotient $A\rightarrow\kf[x]^{gr}/(x^p)$ is conormal, since $\kf[x]^{gr}/(x^p)=\kappa(A)$, but not cocentral. It is not hard to see that this is the only possible candidate for an elementary quotient, and thus that $A$ has no cocentral elementary quotients.
\end{ex}

In the above example, the quotient was conormal, so one somewhat na\"\i vely hopes that this procedure will produce conormal quotients. The problem is that we could have several candidates for $\chi$ and, as examples \ref{ex:oneyes,oneno} and \ref{ex:linearcombination} show, not all choices would work. In \ref{ex:oneyes,oneno}, one of the elements of top degree of a given basis  will not give a conormal quotient, while the other one will. In \ref{ex:linearcombination}, none of the elements of the basis given in the representation will give a conormal quotient, but a specific linear combination will. 

\begin{ex}\label{ex:oneyes,oneno} In this example we show the coordinate ring of a finite gr-group variety where some of the quotients in the highest degree (which guarantees that they are elementary) are conormal and some are not. Let $A=\kf[m,x,y]^{gr}/(m^p,x^p,y^p)$, $\charfield\kf=p$, $|m|=0$, $|x|=|y|=2$, $\Dred(m)=m\otimes m$, $\Dred(x)=\Dred(y)=x\otimes m$, $\e(m)=\e(x)=\e(y)=0$. Notice that $A\rightarrow\kf[x]^{gr}/(x^p)$ is not conormal, but $A\rightarrow\kf[y]^{gr}/(y^p)$ is.
\end{ex}

\begin{ex}\label{ex:linearcombination} This example is a variation of the previous one; it is given in characteristic $2$, but it is easily adaptable to any characteristic.

Let $A=\boldsymbol{\mathrm{F}}_2[m,x,y,z]^{gr}/(m^2,x^2,y^2,z^2)$, $|m|=0$, $|x|=|y|=|z|=2$, $\Dred(m)=m\otimes m$, $\Dred(x)=\Dred(y)=\Dred(z)=x\otimes m+y\otimes m+z\otimes m$, $\e(m)=\e(x)=\e(y)=\e(z)=0$. None of the quotients $A\rightarrow\boldsymbol{\mathrm{F}}_2[x]^{gr}/(x^2)$, $A\rightarrow\boldsymbol{\mathrm{F}}_2[y]^{gr}/(y^2)$, or $A\rightarrow\boldsymbol{\mathrm{F}}_2[z]^{gr}/(z^2)$ are conormal. However, $A\rightarrow\boldsymbol{\mathrm{F}}_2[x+y]^{gr}/((x+y)^2)$ is.
\end{ex}

For completion, we comment on another possible approach, which fails in the general case, but might work in some specific case. This approach is based on the observation that an evenly graded Hopf algebra is a Hopf algebra, and thus one might hope to only have to work with conormal quotient in odd degree and reduce to the evenly graded case. This might seem a particularly interesting approach since the comultiplication of the elements of odd degree is simpler (the square of these elements is zero). Unfortunately, the even degree part might not be a sub-Hopf algebra. For example, in \ref{ex:even-odd}, the only (conormal) quotient is in even degree, while in odd degree there is no quotient. As mentioned, this happens because the even part is not a sub-Hopf algebra. 

\begin{ex}\label{ex:even-odd} This is an example where the even part in not a sub-Hopf algebra. Let $A=\kf[x,y]^{gr}/(x^p)$,  $\charfield\kf=p>2$, $|x|=2$, $|y|=1$, $\Dred(y)=0$, $\Dred(x)=y\otimes y$, $\e(x)=\e(y)=0$. The even part of $A$ is $\kf[x]^{gr}/(x^p)$, which is not a sub-Hopf algebra, and, for the same reason, the \emph{algebra} quotient $A\rightarrow\kf[y]^{gr}$ is not a Hopf algebra quotient. However, there is conormal quotient (cocentral, in fact): $A\rightarrow \kf[x]^{gr}/(x^p)$.
\end{ex}

\begin{qt}\label{qt:conormalkappa} If $A$ is the coordinate ring of finite gr-group variety, does $A$ have an elementary conormal quotient in positive degree?
\end{qt}

Notice that such quotient would necessarily be a quotient of $\kappa(A)$ as well. A positive answer to the above question would provide an alternative proof of the finite generation of the cohomology, using the diagram below:

$$
\xymatrix{ & \kf \ar[d] & \kf \ar[d] & \kf \ar[d] & \\
\kf \ar[r] & A_0 \ar[r]^{=} \ar[d] & A_0 \ar[r] \ar[d] & \kf \ar[r] \ar[d] & \kf\\
\kf \ar[r] & J_E \ar[r] \ar[d] & A \ar[r] \ar[d] & E \ar[r] \ar[d]^{=} & \kf \\
\kf \ar[r] & \kappa(J_E) \ar[r] \ar[d] & \kappa(A) \ar[r] \ar[d] & E \ar[r] \ar[d] & \kf\\
 & \kf & \kf & \kf. & }
$$

It is our belief that a suitable variation of Wilkerson's construction should provide such quotients. In the next section, we will provide an alternative construction, which, however, does not guarantee that the quotients are in positive degree (i.e., they might be in degree $0$).

\subsection{Groups of unitriangular type}

We have the following characterization of conormal quotients of Hopf algebras.

\begin{lm}\label{lm:conormal} Let $A$ be the dual of the coordinate ring of a finite gr-group variety over $\kf$, let $B\subseteq A$ be a graded Hopf sub-algebra and let $I_B$ be its augmentation ideal. Let $A^*$ and $B^*$ be the duals of $A$ and $B$, respectively, and let $A^*\twoheadrightarrow B^*$ be the natural surjection dual of the inclusion $B\subseteq A$.
The following are equivalent:
\begin{enumerate}
\item $I_B\cdot A=A\cdot I_B$;
\item $I_B\cdot A$ is a graded Hopf ideal of $A$;
\item $A\cdot I_B$ is a graded Hopf ideal of $A$;
\item $I_B\cdot A$ is a (graded) Lie ideal of $A$;
\item $A\cdot I_B$ is a (graded) Lie ideal of $A$;
\item there is an exact sequence of graded Hopf algebras
\begin{equation}
\kf\rightarrow B\rightarrow A\rightarrow C_B\rightarrow\kf;
\end{equation}
\item there is an exact sequence of graded Hopf algebras
\begin{equation}
\kf\rightarrow J_{B^*}\rightarrow A^*\rightarrow B^*\rightarrow\kf;
\end{equation}
\item $A_*\hspace{-2ex}\qed_{B^*}\kf=\kf\hspace{-2ex}\qed_{B^*}A^*$.
\end{enumerate}
\end{lm}

\begin{proof} Notice that condition $(a)$ is equivalent to $I_B\cdot A$ being a bi-ideal, see \cite{MM}.

$(a)\Rightarrow(b)$ Since $I_B\cdot A=A\cdot I_B$, $C=A/(I_B\cdot A)$ is a cocommutive bialgebra, which is finite dimensional by the assumption on the dimensions. The dual $C^*$ is a positively graded, commutative bialgebra. Notice that $C^*\subseteq A^*$. In particular, $C^*$ is of finite type and $C^*_0\subseteq A^*_0\cong\kf[x_1,\ldots,x_r]/(x_1^{p^{e_1}},\ldots,x_r^{p^{e_r}})$, \cite[Theorem 14.4]{W}. Therefore $C^*_0$ is a local ring. By \cite[Proposition A.4]{AR1}, $C^*$ is gr-local. Thus, by \cite[Theorem 3.3]{AR1}, $C^*$ is a Hopf algebra. Since $C$ is finite dimensional, $C\cong (C^*)^*$ is also a Hopf algebra.

$(a)\Rightarrow (c)$ The proof is as above.

$(b),(c)\Rightarrow (a)$ Clear.

$(a)\Rightarrow(d)$ Let $a\in A$ and $z\in I_B\cdot A=A\cdot I_B$. Let $z=\sum b_ix_i=\sum y_jc_j$, with $b_i,c_j\in A$, $x_i,y_j\in I_B$; then
\begin{eqnarray*}
[a,z]&=&az-(-1)^{|a||z|}za=a\sum b_ix_i-(-1)^{|a||z|}(\sum y_jc_j)b=\\
&=&\sum ab_ix_i-(-1)^{|a||z|}\sum y_jc_jb.
\end{eqnarray*}

Notice that $\sum ab_ix_i\in A\cdot I_B=I_B\cdot A$, while clearly $\sum y_jc_jb\in I_B\cdot A$. Thus $[a,z]\in I_B\cdot A$.

$(a)\Rightarrow (e)$ As above.

$(d)\Rightarrow (a)$ Let $a\in A$ and $x\in I_B$. Since $1\in A$, $x=x\cdot 1\in I_B\cdot A$; then $[a,x]=ax-(-1)^{|a||x|}xa\in I_B\cdot A$. Therefore $ax=(-1)^{|a||x|}xa-[a,x]\in I_B\cdot A$.

$(e)\Rightarrow (a)$ As above.

$(b)\Rightarrow(f)$ Set $C_B=A/(I_B\cdot A)$.

$(f)\Rightarrow(a)$ It is enough to observe that $I_B\cdot A=\ker(A\rightarrow C_B)=A\cdot I_B$.

$(f)\Leftrightarrow(g)$ This is just duality.

$(g)\Leftrightarrow(h)$ This is as in \cite{MM} since the surjection $A^*\rightarrow B^*$ is split as map of $\kf$-vector spaces. We point out that, in this case, $J_{B^*}=A_*\hspace{-2ex}\qed_{B^*}\kf$.

\end{proof}

\begin{lm}[Noether correspondence for graded group schemes]\label{lm:Noether} Let $G$ and $H$ be graded group schemes, with $H$ a normal closed graded subgroup scheme of $G$. There is a natural bijection
$$
\left\{\begin{array}{c}\textrm{closed graded subgroup}\\\textrm{schemes of $G/H$}\end{array}\right\}\leftrightarrow\left\{\begin{array}{c}\textrm{closed graded subgroup}\\\textrm{schemes of $G$ containing $H$}\end{array}\right\}
$$
which preserves normality.
\end{lm}

\begin{proof} This is immediate from the functorial description, lemma \ref{lm:Yoneda}.
\end{proof}

\begin{rk} The previous lemma can be proven working directly with Hopf algebras. By hypothesis, there is a short exact sequence of graded Hopf algebras
$$
\kf\rightarrow\kf[G/H]\rightarrow\kf[G]\rightarrow\kf[H]\rightarrow\kf.
$$
Notice that $K$ is a closed graded subgroup scheme of $G$ containing $H$ if and only if the surjection $\kf[G]\rightarrow\kf[H]$ factors as $\kf[G]\rightarrow\kf[K]\rightarrow\kf[H]$. In that case, the group $K/H$ is represented by the preimage of $\kf[K]$ in $\kf[G/H]$. Conversely, if $K$ is a subgroup of $\kf[G/H]$, the corresponding subgroup of $G$ containing $H$ is represented by $\kf[G]\otimes_{\kf[G/H]}\kf[K]$.
\end{rk}

\begin{df} Let $I=(I_1,\ldots,I_n)$ be a collection of non-negative integers, with $I_j\leq I_{j+1}$. The \emph{graded unitriangular scheme $\grUT_I$} is the gr-group scheme with coordinate algebra $\kf[x_{ij}]^{gr}_{1\leq i<j\leq n}$, $|x_{ij}|=I_j-I_i$, $\e(x_{ij})=0$ and $\Dred(x_{ij})=\sum_{i<k<j}x_{ik}\otimes x_{kj}$.
\end{df}

As usual, if $I_j-I_i$ is odd and $\charfield\kf\neq2$, then $x_{ij}^2=0$.

\begin{rk} When $I=(0,\ldots,0)$, this is the usual unitriangular group scheme, i.e., the scheme of upper triangular matrices with $1$s on the main diagonal.
\end{rk}

\begin{rk} The gr-group scheme $\grUT_I$ is naturally a gr-subgroup scheme of the graded general linear group $\grGL_I$, as defined in \cite[Definition 5.1]{AR2}.  Moreover, it is positively graded.
\end{rk}
%
%\begin{df} Let $I=(I_1,\ldots,I_n)$ be a collection of non-negative integers with $I_j\leq I_{j+1}$; let $\Pscr\subseteq\{(i,j)\,|\,1\leq i<j\leq n\}$ and let $\Escr=\{e_{ij}\,|\,(i,j)\in\Pscr\}$. The \emph{graded unitriangular scheme of type $(\Pscr,\Escr)$ $\grUT_I^{\Pscr,\Escr}} is the gr-group scheme represented by $\kf[x^{p^{e_{ij}}}_{ij}]^{gr}_{(i,j)\in\Pscr, e_{ij}\in\Escr}/(x_{ij}^{p^r})$
%\end{df}

\begin{lm}\label{lm:Ui} For every $I$ and $r\geq1$, the gr-group variety $(\grUT_I)_{(r)}$ has a chain of closed graded subgroup varieties
$$
(\grUT_I)_{(r)}=U_{s+1}>U_s>\ldots>U_1>U_0=\{1\},
$$
such that each $U_i$ is normal in $(\grUT_I)_{(r)}$ and, for each $0\leq i\leq s$, $U_{i+1}/U_i$ is an elementary gr-group scheme.
\end{lm}

\begin{proof} We will construct this sequence inductively: by the Noether correspondence, Lemma \ref{lm:Noether}, it is enough to show that for each $0\leq i\leq s$, $(\grUT_I)_{(r)}/U_i$ has a normal elementary graded subgroup scheme.

Let us assume that we have constructed the first $i$ subgroups and that the coordinate ring of $U^i=(\grUT_I)_{(r)}/U_i$ is of the form
\begin{equation}\label{eq:k[U]}
\kf[U^i]=\frac{\kf[x_{ij},x^{p^e}_{ml}]^{gr}_{(i,j)\in\Pscr}}{(x_{ij}^{p^r},x_{ml}^{p^r})},\Pscr=\{(i,j)\,|\,1\leq i<j<l\}\cup\{(i,l)\,|\,m< i<l\},\;e<r.
\end{equation}

As usual, we are being redundant in the notation, since some of the variables might satisfy $x_{ij}^2=0$. This is essentially the graded unitriangular matrices, minus some variables which might have already disappeared, starting from the top right corner and proceeding top to bottom, right to left. In down to earth terms, the pair $(m,l)$ corresponds to the first existing row in the last existing column. The comultiplication is of the form
\begin{equation}\label{eq:DeltaK[U]}
\Dred(x_{ij})=\sum_{i<k<j}x_{ik}\otimes x_{kj},\quad \Dred(x_{ml}^{p^e})=\sum_{m<k<l}x_{mk}^{p^e}\otimes x_{kl}^{p^e}.
\end{equation}
Notice that this is certainly the case when $i=0$, as $U^0=(\grUT_I)_{(r)}$.

Let
$$
\alpha:\kf[U^i]\rightarrow\kf[E]=\kf[t]^{gr}/(t^{p^r}),\quad x_{ml}^{p^e}\mapsto t, x_{ij}\mapsto0,
$$
where $|t|=p^e(I_l-I_m)=|x_{ml}^{p^e}|$, $\Dred(t)=0$, $\e(t)=0$. Again, if $|x_{ml}|$ is odd, then $e=0$ and the map is to $\kf[E]=\kf[t]^{gr}=\kf[t]^{gr}/(t^2)$. Since the element $x_{ml}$ does not appear in the comultiplication of any other $x_{ij}$ in $\kf[U^i]$, the above map is a quotient of Hopf algebras. By Lemma \ref{lm:conormal}, $E$ is conormal in $U^i$ if and only
$$
\kf[U^i]\hspace{-2ex}\qed_{\kf[E]}\kf=\kf\hspace{-2ex}\qed_{\kf[E]}\kf[U^i].
$$
We notice that, again since $x_{ml}$ does not appear in the comultiplication of any other $x_{ij}$, we have $x_{ij}\in\kf[U^i]\hspace{-2ex}\qed_{\kf[E]}\kf$: indeed
$$
(id\otimes\alpha)(\Delta(x_{ij}))=(id\otimes\alpha)(x_{ij}\otimes 1+1\otimes x_{ij}+\sum x_{ik}\otimes x_{kj})= x_{ij}\otimes 1=(id\otimes 1)(x_{ij}).
$$
Similarly
$$
(id\otimes\alpha)(\Delta(x_{ml}^{p^{e+1}}))=x_{ml}^{p^{e+1}}\otimes 1=(id\otimes 1)(x_{ml}^{p^{e+1}}).
$$
Thus $\kf[U^i]\hspace{-2ex}\qed_{\kf[E]}\kf$ contains the subalgebra generated by $x_{ij}$ and $x_{ml}^{p^{e+1}}$. Since
$$
\dim_{\kf}\kf[U^i]\hspace{-2ex}\qed_{\kf[E]}\kf=\frac{\dim_{\kf}\kf[U^i]}{\dim_{\kf}\kf[E]}=\frac{\dim_{\kf}\kf[U^i]}{p}=\dim_{\kf}\frac{\kf[x_{ij},x_{ml}^{p^{e+1}}]^{gr}_{(i,j)\in\Pscr}}{(x_{ij}^{p^r},x_{ml}^{p^r})},
$$
we must have
$$
\kf[U^i]\hspace{-2ex}\qed_{\kf[E]}\kf=\frac{\kf[x_{ij},x_{ml}^{p^{e+1}}]^{gr}_{(i,j)\in\Pscr}}{(x_{ij}^{p^r},x_{ml}^{p^r})}.
$$
Similarly
$$
\kf\hspace{-2ex}\qed_{\kf[E]}\kf[U^i]=\frac{\kf[x_{ij},x_{ml}^{p^{e+1}}]^{gr}_{(i,j)\in\Pscr}}{(x_{ij}^{p^r},x_{ml}^{p^r})}.
$$
Of course, if $e+1=r$, the variable $x_{ml}$ will not be in the kernel anymore. If the degree of $x_{ml}$ is odd, the computation is the same, but, instead of dividing by $p$ in the dimension, we divide by $2$. Thus we have shown that there is a short exact sequence of graded Hopf algebras
$$
\kf\rightarrow\frac{\kf[x_{ij},x_{ml}^{p^{e+1}}]^{gr}_{(i,j)\in\Pscr}}{(x_{ij}^{p^r},x_{ml}^{p^r})}\rightarrow\kf[U^i]\rightarrow\kf[E]\rightarrow\kf,
$$
with $\kf[E]$ elementary, which also shows inductively the description in \eqref{eq:k[U]}, thus completing the proof.
\end{proof}

We say that a graded group scheme as in \eqref{eq:k[U]} and \eqref{eq:DeltaK[U]} is \emph{of unitriangular type}.

\begin{lm} Let $G$ be a finite gr-group variety. There exists a closed gr-subgroup embedding $G\hookrightarrow(\grUT_I)_{(r)}$ for some index $I$ and some positive integer $r$.
\end{lm}

\begin{proof} The proof is the same as \cite[Proposition 5.5]{AR2}.
\end{proof}

\begin{prop}\label{prop:conormalmonogenicquotients} Let $G$ be a finite gr-group variety. Then $G$ has a closed normal elementary graded subgroup scheme.
\end{prop}

\begin{proof} Let $G\hookrightarrow(\grUT_I)_{(r)}$ and let $(\grUT_I)_{(r)}=U_{s+1}>U_s>\ldots>U_1>U_0=\{1\}$ be a chain of subgroups as in Lemma \ref{lm:Ui}. Let us consider the new chain
\begin{equation}\label{eq:chain}
G=G\cap U_{s+1}\geq G\cap U_s\geq\ldots\geq G\cap U_1\geq G\cap U_0=\{1\}.
\end{equation}
Notice that each $G\cap U_i$ is a normal closed graded subgroup scheme of $G$. Moreover, there is a natural map $G\cap U_i\rightarrow U_i\rightarrow U_i/U_{i-1}$. The kernel of this map is precisely $G\cap U_{i-1}$, and we obtain an injective morphism of graded group schemes
$$
\frac{G\cap U_i}{G\cap U_{i-1}}\hookrightarrow\frac{U_i}{U_{i-1}}.
$$
Since $E=U_i/U_{i-1}$ is elementary, $(G\cap U_i)/(G\cap U_{i-1})\cong E$ or $(G\cap U_i)/(G\cap U_{i-1})\cong\{1\}$. Since $G\cap U_{s+1}=G$ and $G\cap U_0=\{1\}$, these quotients cannot be all $\{1\}$. Let $i_0$ be the smallest index such that $(G\cap U_i)/(G\cap U_{i-1})\cong E$. Then, for each $i<i_0$, $G\cap U_i=\{1\}$ and
$$
E\cong \frac{G\cap U_{i_0}}{G\cap U_{i_0-1}}\cong G\cap U_{i_0};
$$
therefore $G\cap U_{i_0}$ is an elementary normal closed graded subgroup scheme of $G$.
\end{proof}

\begin{rk} The above \eqref{eq:chain} is a version of \cite[Lemma A.15]{HS} in the non-connected case.
\end{rk}

\section{Cohomology of finite graded group varieties}

We can now prove our main theorem. We will prove it by induction, using Proposition \ref{prop:conormalmonogenicquotients}.

\begin{lm}\label{invact}
Let $G$ be a finite gr-group variety that acts on a gr-commutative ring $R$. There exists $N$ such that, for all $x\in R$, the element $x^{p^N}$ is invariant under $G$. 
\end{lm}

\begin{proof}
Having  $G$ acting on $R$ corresponds to a coaction of $A = \kf[G]$ on $R$.  Since $G$ is a gr-group variety, we write $(A_0, \mathfrak{m})$ for the local ring $A_0$. For simplicity let $x$ be a homogeneous element in $R$, then the coaction is of the form
 $$\Delta_R(x) = (1+m) \otimes x + \sum m_1 \otimes x_1 + \sum a_1 \otimes b_2,$$
 where $m \in \mathfrak{m}$, $|x_1| = |x|$ and $|a_1| + |b_2|$ with $|b_2| < |x|$. Since $A$ and $A_0$ are finite dimensional and $A_0$ is local, it follows that $\Delta_R(x^{p^N}) = 1 \otimes x^{p^N}$ for $N >> 0$, namely, as soon as $\mathfrak{m}^{p^N}=0$. 
\end{proof}

%We also recall the following result.
%
%\begin{lm}[{\cite[Lemma A.6]{HS}}] Let $\{E_r,d_r\}$ be a spectral sequence of finitely generated modules over a Noetherian ring $T$. There is an integer $N$ with the property that all the differentials $d_r$ are zero if $r>N$.
%\end{lm}

\begin{thm}\label{thm:cohom1} Let $A$ be the coordinate ring of a finite gr-group variety over a field $\kf$. We have the following:
\begin{itemize}
\item[(H)] the cohomology ring $H^{*,*}(A,\kf)$ is a finitely generated $\kf$-algebra;
\item[(Q)] if $A\rightarrow B$ is a quotient of graded commutative Hopf algebras, $H^{*,*}(B,\kf)$ is a finite $H^{*,*}(A,\kf)$-module;
\item[(M)] if $R$ is a Noetherian ring on which $A$ acts trivially and $M$ is a finite $R$-module, $H^{*,*}(A,M)$ is a finite $H^{*,*}(A,R)$-module.
\end{itemize}
\end{thm}

\begin{proof} The proof will proceed by induction on $\dim_{\kf}A$. For each of the above properties $(P)$, we will denote by $(P_n)$ the correspond property for all $A$ with $\dim_{\kf}A=n$.

 If $A$ is elementary, then the results are known -- see \cite{L2} or \cite{Wi}.

Let $\dim_{\kf}A=n$, $A$ be non-elementary, and let $E$ be a conormal elementary quotient of $A$ which exists by Proposition \ref{prop:conormalmonogenicquotients}; let $J_E$ be the kernel of $A\rightarrow E$, so that we have the short exact sequence of graded Hopf algebras:
\begin{equation}\label{eq:E}
\kf\rightarrow J_E\rightarrow A\rightarrow E\rightarrow\kf,
\end{equation}
where $\dim_{\kf} E=e$, with $e=2$ or $e=p$, and $\dim_{\kf}J_E=\dim_{\kf}A/\dim_{\kf}E=n/e$. The strategy of the proof will be the following: $(H_{n/e})+(M_{n/e})\Rightarrow(H_n)$; $(H_{n/e})+(M_{n/e})\Rightarrow(M_n)$; and $(H_{n/e})+(M_{n/e})+(Q_{n/e})\Rightarrow(Q_n)$. 

\vspace{2ex}

{$(H_{n/e})+(M_{n/e})\Rightarrow(H_n)$}

Let
$$
E_2^{*,*}=H^{*,*}(J_E,H^{*,*}(E,\kf))\Rightarrow H^{*,*}(A,\kf)
$$
be the LHS spectral sequence induced by the short exact sequence in \eqref{eq:E}, which exists because of Theorem \ref{thm:LHS}. By Lemma \ref{invact} there is $N\gg0$ such that, for all $x\in H^{*,*}(E,\kf)$, $x^{p^N}$ is $J_E$ invariant. Moreover, by \cite[Proposition 1.4.11]{Pal}, a high enough power of the polynomial generator of $H^{*,*}(E,\kf)$ is a permanent cycle. Let
\begin{equation}\label{eq:S}
S=H^{*,*}(E,\kf)^{p^N},
\end{equation}
for $N\gg0$. Notice that $S$ is Noetherian (since it is isomorphic to a polynomial ring in one variable), $J_E$-invariant by construction and of permanent cycles. Since $H^{*,*}(E,\kf)$ is finite over $S$, by $(M_{n/e})$ we have that $H^{*,*}(J_E,H^{*,*}(E,\kf))$ is finite over 
\begin{equation}\label{eq:T}
H^{*,*}(J_E,S)\cong H^{*,*}(J_E,\kf)\otimes S.
\end{equation}
Since $H^{*,*}(J_E,\kf)\subseteq E_2^{*,0}$ in the LHS spectral sequence, each page is naturally a module over it. Thus, $T=H^{*,*}(J_E,\kf)\otimes S$ is made of permanent cycles, it is Noetherian by $(H_{n/e})$ and the $E_2$ page of the LHS spectral sequence is a finite module over $T$. By \cite[Lemma 3.1]{Wi}, $H^{*,*}(A,\kf)=E^{*,*}_\infty$ is Noetherian.

\vspace{2ex}

{$(H_{n/e})+(M_{n/e})\Rightarrow(M_{n/e})$}

Let
$$
F_2^{*,*}=H^{*,*}(J_E,H^{*,*}(E,M))\Rightarrow H^{*,*}(A,M)
$$
and let
\begin{eqnarray*}
&&{E'}_2^{*,*}=H^{*,*}(J_E,H^{*,*}(E,R))=H^{*,*}(J_E,H^{*,*}(E,\kf))\otimes R\Rightarrow\\
&&\Rightarrow H^{*,*}(A,R)\cong H^{*,*}(A,\kf)\otimes R.
\end{eqnarray*}
Let $S$ and $T$ be the rings constructed in \eqref{eq:S} and \eqref{eq:T}, and let $S'=S\otimes R$ and $T'=T\otimes R$; by what was proven in the previous implication, $E'_2$ is a finite $T'$-module. On the other hand $H^{*,*}(E,M)$ is finite over $H^{*,*}(E,R)=H^{*,*}(E,\kf)\otimes R$. Since $H^{*,*}(E,\kf)$ is finite over $S$, $H^{*,*}(E,M)$ is finite over $S'$; since $S$ is $J_E$-invariant, $F_2$ is finite over $S'$, which is a subring of invariants of $E'_2$ (like in the proof of the previous implication). Moveover, $T'$ is made of permanent cycles, and thus each page $F_r$ is finite over $E'_r$, which is finite over a Noetherian ring $T'$ of permanent cycles. Since $T'$ is Noetherian, so is ${T'}^{ev}$ (its even part), and $E'_r$ is still finite over ${T'}^{ev}$. By \cite[Lemma 1.6]{FS}, $F^{*,*}_\infty=H^{*,*}(A,M)$ is finite over ${E'}^{*,*}_\infty=H^{*,*}(A,R)$.

\vspace{2ex}

{$(H_{n/e})+(M_{n/e})+(Q_{n/e})\Rightarrow(Q_n)$}

We have the following diagram

$$
\xymatrix{\kf \ar[r] & J_E \ar[r] \ar[d] & A \ar[r] \ar[d] & E \ar[r] \ar[d] & \kf \\
\kf \ar[r] & J_{E'} \ar[r] \ar[d] & B \ar[r] \ar[d] & E'=B\otimes_AE \ar[r] \ar[d] & \kf\\
 & \kf & \kf & \kf. &}
$$
In term of graded group schemes, if $G_A$ is the gr-group scheme represented by $A$ (and so forth) it corresponds to
$$
\xymatrix{\{1\} \ar[r] & G_E \ar[r] & G_A \ar[r] & G_A/G_E \ar[r] & \{1\}\\
\{1\} \ar[r] & G_E\cap G_B \ar[r] \ar@{^(->}[u] & G_B \ar[r] \ar@{^(->}[u] & G_B/G_E\cap G_B \ar[r] \ar@{^(->}[u] & \{1\}.  }
$$

Since $B\otimes_AE$ is a quotient of $E$, that is elementary, there are only two possibilities. Either $E'\cong\kf$ or $E'\cong E$ (corresponding to the two possibilities $G_B\cap G_E=\{1\}$ or $G_B\cap G_E=E$). If $E'\cong\kf$, then $J_{E'}\cong B$, which means that $B$ is a quotient of a Hopf algebra of lower dimension. By $(Q_{n/e})$, the map $H^{*,*}(J_E,\kf)\rightarrow H^{*,*}(B,\kf)$ is finite. Since this map factors as $H^{*,*}(J_E,\kf)\rightarrow H^{*,*}(A,\kf)\rightarrow H^{*,*}(B,\kf)$, the map $H^{*,*}(A,\kf)\rightarrow H^{*,*}(B,\kf)$ is finite.

If $E'\cong E$, let
$$
E_2^{*,*}=H^{*,*}(J_E,H^{*,*}(E,\kf))\Rightarrow H^{*,*}(A,\kf)
$$
and let 
$$
{F'}_2^{*,*}=H^{*,*}(J_{E'},H^{*,*}(E',\kf))\Rightarrow H^{*,*}(B,\kf).
$$
Let $S$ and $T$ be the rings constructed in \eqref{eq:S} and \eqref{eq:T}. Using $(H_{n/e})$ and $(M_{n/e})$, as in the proof of the previous implication, we have that $T^{ev}$ is a Noetherian ring of permanent cycles and that $E_2^{*,*}$ is finite over it. Similarly, $H^{*,*}(J_{E'},H^{*,*}(E',\kf))$ is finite over $H^{*,*}(J_{E'},S)\cong H^{*,*}(J_{E'},\kf)\otimes S$. Notice that we are identifying $H^{*,*}(E,\kf)$ and $H^{*,*}(E',\kf)$, since $E\cong E'$. On the other hand, by $(Q_{n/e})$, $H^{*,*}(J_{E'},\kf)$ is finite over $H^{*,*}(J_E,\kf)$. Thus, ${F'}_2^{*,*}$ is finite over the subring of permanent cycles $T^{ev}$ of $E_2^{*,*}$, which means that ${F'}_r^{*,*}$ is always finite over $E_r^{*,*}$. By \cite[Lemma 1.6]{FS}, ${F'}^{*,*}_\infty=H^{*,*}(B,\kf)$ is finite over $E^{*,*}_\infty=H^{*,*}(A,\kf)$.
\end{proof}

\bibliographystyle{amsalpha}

\begin{thebibliography}{AR14b}

\bibitem[AR14a]{AR2}
C.~I. Aponte~Rom\'{a}n, \emph{Graded 1-parameter subgroups and detection
  properties}, arXiv:1412.0232 (2014).

\bibitem[AR14b]{AR1}
\bysame, \emph{Graded group schemes and graded group varieties},
  arXiv:1405.4026 (2014).

\bibitem[FS97]{FS}
E.~M. Friedlander and A.~Suslin, \emph{Cohomology of finite group schemes over
  a field}, Inventiones mathematicae \textbf{127} (1997), no.~2, 209--270.

\bibitem[HS98]{HS}
M.~J. Hopkins and J.~H. Smith, \emph{Nilpotence and stable homotopy theory ii},
  Annals of mathematics (1998), 1--49.

\bibitem[Liu62]{L2}
A.~L. Liulevicius, \emph{The factorization of cyclic reduced powers by
  secondary operations}, Memoirs of the American Mathematical Society
  \textbf{104} (1962), no.~3, 443--449.

\bibitem[MM65]{MM}
J.W. Milnor and J.C. Moore, \emph{On the structure of {H}opf algebras},
  Princeton University Press, 1965.

\bibitem[Pal01]{Pal}
J.~H. Palmieri, \emph{Stable homotopy over the {S}teenrod algebra}, vol. 716,
  American Mathematical Soc, 2001.

\bibitem[ST12]{ST}
K.~Schwede and K.~Tucker, \emph{A sruvery of test ideals}, Progress in
  Commutative Algebra 2 (2012), 39--99.

\bibitem[Wat79]{W}
W.C. Waterhouse, \emph{Introduction to affine group schemes}, vol.~66,
  Springer, 1979.

\bibitem[Wil81]{Wi}
C.~Wilkerson, \emph{The cohomology algebras of finite dimensional {H}opf
  algebras}, American Mathematical Society \textbf{264} (1981), no.~1.

\end{thebibliography}

\end{document}